\documentclass[a4paper]{amsart}

\usepackage[T1]{fontenc}
\usepackage[latin1]{inputenc}
\usepackage{amsfonts}
\usepackage{amssymb}
\usepackage{amsxtra}
\usepackage{amsthm}
\usepackage{ae}
\usepackage[all]{xy}
\usepackage{feyn}

\renewcommand{\a}{\alpha}
\renewcommand{\b}{\beta}
\newcommand{\e}{\epsilon}

\renewcommand{\l}{\lambda}
\newcommand{\g}{\gamma}
\renewcommand{\d}{\delta}

\newcommand{\Th}{\Theta}
\newcommand{\oc}{c^\ast}

\DeclareMathOperator{\sgn}{sgn}
\DeclareMathOperator{\ind}{Index}
\DeclareMathOperator{\tr}{Tr}
\DeclareMathOperator{\res}{Res}

\newcommand{\I}{\mathcal{I}}
\newcommand{\D}{\mathcal{D}}
\newcommand{\AD}{\mid \mathcal{D} \mid}
\renewcommand{\H}{\mathcal{H}}
\renewcommand{\S}{\mathcal{S}}
\newcommand{\A}{\mathcal{A}}
\newcommand{\B}{\mathcal{B}}
\newcommand{\NH}{\mathbb{N}}
\newcommand{\ZH}{\mathbb{Z}}
\newcommand{\CH}{\mathbb{C}}
\newcommand{\RH}{\mathbb{R}}
\renewcommand{\TH}{\mathbb{T}}

\newcommand{\torus}{C^{\infty}(\mathbb{T}^3_\Theta)}

\newcommand{\lcom}{\left[}
\newcommand{\rcom}{\right]}
\newcommand{\ot}{\otimes}
\newcommand{\set}[1]{\lbrace #1 \rbrace}
\newcommand{\ov}[1]{\overline{#1}}

\newcommand{\fl}{\text{for all }}
\newcommand{\braket}[2]{ \langle #1,#2 \rangle} 

\def\build#1_#2^#3{\mathrel{
\mathop{\kern 0pt#1}\limits_{#2}^{#3}}}

\newbox\ncintdbox \newbox\ncinttbox
\setbox0=\hbox{$-$}
\setbox2=\hbox{$\displaystyle\int$}
\setbox\ncintdbox=\hbox{\rlap{\hbox
    to \wd2{\kern-.1em\box2\relax\hfil}}\box0\kern.1em}
\setbox0=\hbox{$\vcenter{\hrule width 4pt}$}
\setbox2=\hbox{$\textstyle\int$}
\setbox\ncinttbox=\hbox{\rlap{\hbox
    to \wd2{\kern-.05em\box2\relax\hfil}}\box0\kern.1em}
\newcommand{\ncint}{\mathop{\mathchoice{\copy\ncintdbox}%
    {\copy\ncinttbox}{\copy\ncinttbox}%
    {\copy\ncinttbox}}\nolimits}

\newcommand{\ra}{\rightarrow}

\newcommand{\nn}{\nonumber}

\numberwithin{equation}{section}
\theoremstyle{plain}
\newtheorem{thm}{Theorem}[section]
\newtheorem{lem}[thm]{Lemma}

\newtheorem*{1_ax}{Dimension}
\newtheorem*{2_ax}{Regularity}
\newtheorem*{3_ax}{Dimension Spectrum}

\theoremstyle{definition}
\newtheorem{defn}{Definition}[section]

\theoremstyle{remark}
\newtheorem{rem}{Remark}[section]

\begin{document}

\title[Chern-Simons action]{Chern-Simons theory for the noncommutative $3$-torus $\torus$}

\author{Oliver Pfante}

\address{Oliver Pfante \\
         Mathematisches Institut\\
         Westfälische Wilhelms-Universität Münster \\
         Einsteinstraße 62 \\
         48149 Münster\\
         Germany 
}
\email{oliver.pfante@uni-muenster.de}

\maketitle

\begin{abstract}
We study the Chern-Simons action, which was defined for noncommutative spaces in general by the author, for the noncommutative $3$-torus $\torus$, the universal $C^\ast$-algebra generated by $3$-unitaries. D. Essouabri, B. Iochum, C. Levy, and A. Sitarz constructed a spectral triple for $\torus$. We compute the Chern-Simons action for this noncommutative space. In connection with this computation we calculate the first coefficient in the loop expansion series of the corresponding Feynman path integral with the Chern-Simons action as Lagrangian. The result is independent of the deformation matrix $\Th$ and always $0$.
\end{abstract}

\section{Introduction}

S.-S. Chern and J. Simons uncovered a geometrical invariant which grew out of an attempt to derive a purely combinatorial formula for the first Pontrjagin number of a $4$-manifold. This Chern-Simons invariant turned out to be a higher dimensional analogue of the $1$-dimensional geodesic curvature and seemed to be interesting in its own right \cite{chern_simons}. This classical Chern-Simons invariant has found numerous applications in differential geometry, global analysis, topology and theoretical physics. \\
E. Witten \cite{witten} used the Chern-Simons invariant to derive a $3$-dimensional quantum field theory in order to give an intrinsic definition of the Jones Polynomial and its generalizations dealing with the mysteries of knots in three dimensional space. Witten's approach is motivated by the Lagrangian formulation of quantum field theory, where observables are computed by means of integration of an action functional over all possible physical states modulo gauge transformations -- i. e. one computes the Feynman \textit{path integral}. \\
Knot polynomials deal with topological invariants, and understanding these theories as quantum field theories, as Witten did, involves the construction of theories in which all of the observables are topological invariants. The physical meaning of \textquotedblleft topological invariance\textquotedblright $\,$is \textquotedblleft general covariance\textquotedblright. A quantity that can be computed from a manifold $M$ as a topological space without any choice of a metric is called a \textquotedblleft topological invariant" by mathematicians. To a physicist, a quantum field theory defined on a manifold $M$ without any a priori choice of a metric on $M$ is said to be generally covariant. Obviously, any quantity computed in a generally covariant quantum field theory will be a topological invariant. Conversely, a quantum field theory in which all observables are topological invariants can naturally be seen as a generally covariant quantum field theory. The surprise, for physicists, comes in how general covariance in Witten's three dimensional covariant quantum field theory is achieved. General relativity gives us a prototype for how to construct a quantum field theory with no a priori choice of metric - we introduce a metric, and then integrate over all metrics. Witten constructs an exactly soluble generally covariant quantum field theory in which general covariance is achieved not by integrating over metrics, instead he starts with a gauge invariant Lagrangian that does not contain any metric -- the Chern-Simons invariant, which becomes the \textit{Chern-Simons action}. \\

Let $M$ denote a $3$-dimensional, closed, oriented manifold and $A \in M_N(\Omega^1(M))$ a hermitian $N \times N$-matrix of differential $1$-forms on $M$. The Chern-Simons action for $A$ is given by:
\begin{equation} \label{eq_1.1}
S_{CS}(A) := \dfrac{k}{4 \pi} \int_M \tr \left( A \wedge dA + \dfrac{2}{3} A \wedge A \wedge A \right) 
\end{equation}
for an integer $k \in \ZH$ which is called the \textit{level}. Clearly, $S_{CS}(A)$ does not depend on a metric, and therefore it is a topological invariant. In addition, the Chern-Simons action is gauge invariant. Under gauge transformation 
\begin{equation*}
A \mapsto A^u= uAu^\ast + u d u^\ast
\end{equation*} 
of the $1$ form by a gauge-map $u: M \ra SU(N)$ we have 
\begin{equation*}
S_{CS}(A^u) = S_{CS}(A) + 2 \pi m
\end{equation*}
for an integer $m \in \ZH$, called the winding number of the gauge-map $u$. Next Witten computes the partition function 
\begin{equation} \label{eq_1.2}
Z(k) = \int  D \lcom A \rcom e^{i S_{CS}(A)} 
\end{equation} 
of the action, where we integrate over all possible hermitian $N \times N$-matrices of differential $1$-forms $A$ on $M$ modulo gauge transformation. Note that the integrand is well defined because it does not depend on the special choice of the representative $A$ of the equivalence class $\lcom  A \rcom$ due to the gauge invariance of the action.\\
The partition function $Z(k)$ is a topological invariant of the manifold $M$ in the variable $k$ corresponding to the so called \textit{no knot}-case in \cite{witten}, which is interesting in its own right. There exists a meromorphic continuation of $Z(k)$ in $k$ which was computed in \cite{witten} for some explicit examples. For the $3$-sphere $M = S^3$ and the gauge group $SU(2)$ one gets 
\begin{equation*}
Z(k) = \sqrt{\dfrac{2}{2+k}} \sin \left( \dfrac{\pi}{k + 2} \right) \, .
\end{equation*}
For manifolds of the form $X \times S^1$, for a $2$-dimensional manifold $X$, and arbitrary gauge group $SU(N)$, one obtains $Z(k) = N_X$ as result, with an integer $N_X \in \NH$ depending only on the manifold $X$.\\

In \cite{pfante_chern_simons} a noncommutative analogue of the Chern-Simons action is defined for noncommutative spaces, and gauge invariance of the action for a noncommutative analogue of gauge-mapping is proved. Noncommutative spaces are spectral triples $\left( \A, \H, \D \right)$ consisting of a pre-$C^\ast$-algebra $\A$ of bounded operators on the Hilbert space $\H$ and an unbounded, self-adjoint operator $\D$ on $\H$, with compact resolvent, such that the commutators $\lcom a, \D \rcom$ for $a \in \A$ are bounded. Spectral triples are the starting point for the study of noncommutative manifolds \cite{connes}. One can think of them being a generalization of the notion of ordinary differential manifolds because every spin manifold without boundary can be encoded uniquely by its spectral triple \cite{connes_reconstruction}. In this framework some core structures of topology and geometry, like the local index theorem \cite{connes_index} of A. Connes and H. Moscovici, were extended far beyond the classical scope.\\ 
In the present paper we apply the general results of \cite{pfante_chern_simons} to a special example of noncommutative spaces: the noncommutative $3$-torus $\torus$, the universal $C^\ast$-algebra generated by three unitaries which are subject to some commutation relations encoded by a skew-symmetric $3 \times 3$-matrix $\Th$. One can think of $\torus$ being a generalization of the commutative $\ast$-algebra $C^\infty(\TH^3)$ of smooth functions on the $3$-dimensional torus $\TH^3$. A spectral triple for $\torus$ is constructed and investigated by D. Essouabri, B. Iochum, C. Levy and A. Sitarz in \cite{iochum}. We apply the results of \cite{iochum} in order to compute the Chern-Simons action explicitly. In addition, we introduce a noncommutative replacement of the path integral \eqref{eq_1.2} in order to compute the first coefficient in the Taylor expansion series of the partition function $Z(k)$ in the coupling constant $k^{-1}$. Unfortunately, the result will be independent of the deformation matrix $\Th$ and zero in any case. More interesting results come up if one applies the general results of \cite{pfante_chern_simons} to another well known noncommutative space: the quantum group $SU_q(2)$ \cite{pfante_sphere}. \\

The approach of the paper is as follows. In the second section we make available the tools of noncommutative differential geometry and give an overview of the results in \cite{pfante_chern_simons}.  The third section starts introducing the spectral triple constructed in \cite{iochum} for the noncommutative $3$-torus $\torus$. Subsequently, we compute the noncommutative Chern-Simons action, defined in \cite{pfante_chern_simons} for arbitrary noncommutative spaces, explicitly. In the fourth section we discuss and solve the problems appearing if one deals with a noncommutative replacement of the path integral \eqref{eq_1.2}: we make precise the meaning of the measure $D \lcom A \rcom$, incorporate gauge breaking due to Faddeev-Popov, and compute the first coefficient in the Taylor expansion series of $Z(k)$ in the variable $k^{-1}$. \\

Finally, it is a great pleasure for me to thank my advisor Raimar Wulkenhaar who took care of me and my work for years.
\bigskip 

\section{Chern-Simons action on noncommutative spaces}
 
In this section we give a short overview of \cite{pfante_chern_simons} and some preliminaries about noncommutative differential geometry.

\subsection{Noncommutative differential geometry}

\begin{defn}
A spectral triple $(\A, \H, \D)$ consists of a unital pre-$C^\ast$-algebra $\A$, a separable Hilbert space $\H$ and a densely defined, self-adjoint operator $\D$ on $\H$ such that
\begin{itemize}
 \item there is a faithful representation $\pi: \A \ra \B(\H)$ of the pre-$C^\ast$-algebra $\A$ by bounded operators on $\H$
 \item the commutator $\lcom \D, \pi(a) \rcom $ is densely defined on $\H$ for all $a \in \A$ and extends to a bounded operator on $\H$
 \item the resolvent $(\D - \l)^{-1}$ extends for $\l \notin \RH$ to a compact operator on $\H$.
\end{itemize}
\end{defn}

\begin{rem}
A pre-$C^\ast$-algebra differs from an ordinary $C^\ast$-algebra being only closed with respect to holomorphic functional calculus. 
\end{rem}

The commutative world provides us with a large class of so called commutative spectral triples $(C^{\infty}(M), L^{2}(\S), \D)$ consisting of the Dirac operator acting on the Hilbert space $L^{2}(\S)$ of $L^{2}$-spinors over a closed spin manifold $M$. \\
Next we introduce the noncommutative replacement of $n$-forms. See \cite{voigt}.

\begin{defn}
For $n \geq 0$ let $\Omega^n(\A) = \A \ot \ov \A^{\otimes n}$ the vector space of noncommutative $n$-forms over $\A$, where $\ov \A = \A - \CH 1$ denotes the vector space we obtain from the algebra $\A$ if one takes away the unital element $1 \in \A$.  
\end{defn}

Here we used the notation
\begin{equation*}
\ov \A^{\otimes n} = \ov \A \otimes \cdots \otimes \ov \A
\end{equation*}
to denote the tensor product of $n$ copies of $\ov \A$. Elements of $\Omega^n(\A)$ are written in the more suggestive form $a^0da^1 \cdots da^n$ for $a^0, \ldots, a^n \in \A$, setting $a^0da^1 \cdots da^n = 0$ if $a^i = 1$ for $i \geq 1$. We also write $da^1 \cdots da^n$ if $a^0 = 1$.

\begin{rem}
The notion of noncommutative $n$-forms and their homology theory can be built up for arbitrary algebras, especially non-unital ones. In this case the definition differs slightly from the one we have just given. See \cite{voigt} for more details.
\end{rem}

One can define a $\A$-$\A$-bimodule structure on $\Omega^n(\A)$. Let us first consider the case $n =1$. We define a left $\A$-module structure on $\Omega^{1}(\A)$ by setting 
\begin{equation*}
a \left( a^{0}da^{1} \right)  = aa^{0}da^{1}.
\end{equation*}
A right $\A$-module structure on $\Omega^{1}(\A)$ is defined according to the Leibniz rule $d(ab) = dab+adb$ by 
\begin{equation*}
\left( a^{0}da^{1} \right) a=a^{0}d \left( a^{1}a \right) -a^{0}a^{1}da \, .
\end{equation*}
With these definitions $\Omega^{1}(\A)$ becomes an $\A$-$\A$-bimodule. It is also easy to verify that there is a natural isomorphism
\begin{equation*}
\Omega^{n}(\A) \cong \Omega^{1}(\A) \otimes_{\A} \Omega^{1}(\A) \otimes_{\A} \cdots \otimes_{\A} \Omega^{1}(\A) = \Omega^{1}(\A)^{\otimes_{\A}n}
\end{equation*} (as vector spaces) for every $n \geq 1$. As a consequence, the spaces $\Omega^{n}(\A)$ are equipped with an $\A$-$\A$-bimodule structure in a natural way. Explicitly, the left $\A$-module structure on $\Omega^{n}(\A)$ is given by
\begin{equation*}
a \left( a^{0} da^{1} \cdots da^{n} \right) = aa^{0}da^{1}\cdots da^{n} 
\end{equation*} 
and the right $\A$-module structure may be written as 
\begin{align*} 
 \left( a^{0} da^{1} \cdots d a^{n} \right) a = \, & a^{0} d a^{1} \cdots d \left(  a^{n}  a \right) + \\
						& \sum^{n-1}_{j=1} (-1)^{n-j} a^{0} da^{1} \cdots d(a^{j}a^{j+1}) \cdots da^{n}da \\
						& + (-1)^{n} a^{0}a^{1}da^{2}\cdots da^{n}da.
\end{align*}
Moreover, we view $\Omega^{0}(\A) = \A$ as an $\A$-bimodule in the obvious way using the multiplication in $\A$. We are also able to define a map $\Omega ^{n}(\A) \otimes \Omega^{m}(\A) \ra \Omega^{n+m}(\A)$ by considering the natural map
\begin{equation*}
\Omega^{1}(\A)^{\otimes_{\A}n} \otimes \Omega^{1}(\A)^{\otimes_{\A}m} \rightarrow \Omega^{1}(\A)^{\otimes_{\A}n} \otimes_{\A} \Omega^{1}(\A)^{\otimes_{\A}m} = \Omega^{1}(\A)^{\otimes_{\A}n+m}.
\end{equation*}
Let us denote by $\Omega(\A)$ the direct sum of the spaces $\Omega^{n}(\A)$ for $n \geq 0$. Then the maps $\Omega^{n}(\A) \otimes \Omega^{m}(\A) \rightarrow \Omega ^{n+m}(\A) $ assemble to the map $\Omega(\A) \otimes \Omega(\A) \rightarrow \Omega(\A)$. In this way the space $\Omega(\A)$ becomes an algebra. Actually $\Omega(\A)$ is a graded algebra if one considers the natural grading given by the degree of a differential form. Let us now define a linear operator $d: \Omega^{n}(\A) \ra \Omega^{n+1}(\A)$ by 
\begin{equation*}
 d \left( a^{0}da^{1}\cdots da^{n} \right) = d a^{0} \cdots d a^{n}, \quad d \left( d a^{1} \cdots d a^{n} \right) = 0
\end{equation*}
for $a^0, \ldots, a^n \in \A$. It follows immediately from the definition that $d^{2}= 0$. A differential form in $\Omega(\A)$ is called homogeneous of degree $n$ if it is contained in the subspace $\Omega^{n}(\A)$. Then the graded Leibniz rule
\begin{equation*}
d(\omega \eta) = d\omega \eta + (-1)^{\mid \omega \mid} \omega d \eta
\end{equation*} 
for homogeneous forms $\omega$ and $\eta$ holds true. \\
We define the linear operator $b: \Omega^{n}(\A) \ra \Omega^{n-1}(\A)$ by 
\begin{align*}
 b \left( a^{0}da^{1}\cdots da^{n} \right) =& (-1)^{n-1} \left( a^{0} da^{1} \cdots d a^{n-1} a^{n} - a^{n}a^{0} da^{1} \cdots d a^{n-1} \right) \\
					   =& (-1)^{n-1} \left[ a^{0}da^{1} \cdots da^{n-1}, a^{n} \right] \\
					   =& a^{0}a^{1}da^2 \cdots da^{n} + \\
						&\sum^{n-1}_{j=1} (-1)^{j} a^{0}da^{1} \cdots d \left( a^{j}a^{j+1} \right) \cdots da^{n} \\
						&+  (-1)^{n} a^{n}a^{0}da^{1} \cdots d a^{n-1}.
\end{align*}
A short calculation shows $b^{2} = 0$. Next we introduce the $B$-operator $B: \Omega^{n}(\A) \ra \Omega ^{n+1} (\A)$ given by
\begin{equation*}
B \left( a^{0}da^{1} \cdots d a^{n} \right) = \sum ^{n}_{i = 0} (-1)^{ni} d a^{n+1-i} \cdots da^{n}da^{0} \cdots da^{n-i}.
\end{equation*}
One can check that $ b^2 = B^{2} = 0$ and $Bb + bB = 0$. According to these equations we can form the $(B,b)$-bicomplex of $\A$.
\begin{displaymath}
\begin{xy}
 \xymatrix{ \vdots \ar[d] & \vdots \ar[d] & \vdots \ar[d] & \vdots \ar[d] \\
	\Omega^{3}(\A) \ar[d]_{b} & \Omega^{2}(\A) \ar[l]_{B} \ar[d]_{b} & \Omega^{1}(\A) \ar[l]_{B} \ar[d]_{b} & \Omega^{0}(\A) \ar[l]_{B} \\
	\Omega^{2}(\A) \ar[d]_{b} & \Omega^{1}(\A) \ar[d]_{b} \ar[l]_{B} & \Omega^{0}(\A) \ar[l]_{B}  \\
	\Omega^{1}(\A) \ar[d]_{b} & \Omega^{0}(\A) \ar[l]_{B}  \\
	\Omega^{0}(\A)  \\
}
\end{xy}
\end{displaymath}
The \textit{cyclic homology} of $\A$ is the homology of the total complex of the $(B,b)$-bicomplex of $\A$. \\
The corresponding cohomological version, cyclic cohomology, is quite easy to define. If $V$ is a vector space we denote by $V' = Hom(V, \mathbb{C})$ its dual space. If $f: V \rightarrow W $ is a linear map then it induces a linear map $W' \rightarrow V'$ which will be denoted by $f'$. Applying the dual space functor to the $(B,b)$-complex we have the $(B',b')$-complex of an algebra $\A$ which is again a bicomplex, i. e. $b'^2 =0$, $B'^2=0$, and $B'b'+b'B'=0$. The \textit{cyclic cohomology} of $\A$ is the cohomology of the total complex of the dual $(B',b')$-bicomplex of $\A$. In the sequel we simply write $B$ instead of $B'$ and $b$ instead of $b'$ because it is always clear which kind of map, the ordinary or the dual one, is considered.\\

Next we discuss the noncommutative replacement of pseudo differential calculus and the Atiyah-Singer's index theorem. The main source for this part is \cite{connes_index}. In order to make pseudo differential calculus work for spectral triples we must impose three additional constraints. 

\begin{1_ax}
There is an integer $n$ such that the decreasing sequence $(\l_{k})_{k \in \mathbb{N}}$ of eigenvalues of the compact operator $\mid \D \mid ^{-1} $ satisfies
\begin{equation*}
\lambda_{k} = O \left( k^{-1/n} \right)
\end{equation*} when $k \ra \infty $.
\end{1_ax}

The smallest integer which fulfils this condition is called the \textit{dimension} of the spectral triple. In the case of a $p$-dimensional, closed spin manifold, the dimension of the triple $(C^{\infty}(M), L^{2}(\S), \D)$ coincides with the dimension $p$ of the manifold. Since $\D$ has compact resolvent, its kernel $\ker \D$ is finite dimensional and $\AD ^{-1}$, with $\AD = \sqrt{\D^{2}}$, is well defined as $\AD ^{-1}$ on the orthogonal complement of $\ker \D$ and $0$ on $\ker \D$. Since the kernel finite dimensional, it does not influence the asymptotic behaviour of the Dirac operator or its inverse respectively. \\
Now we consider the derivation $\d$ defined by $\d (T) = \lcom \AD, T \rcom $ for an operator $T$ on $\H$.

\begin{2_ax}
Any element $b$ of the algebra generated by $\pi(\A)$ and $\lcom \D, \pi(\A) \rcom $ is contained in the domain of $\d^k$ for all $k \in \NH$, i. e. $\delta^{k}(b)$ is densely defined and has a bounded extension on $\H$.
\end{2_ax}

In the case of a spin manifold $M$ this condition tells us that we should work with $C^{\infty}(M)$-functions only. Combining the dimension and regularity conditions we obtain that the functions 
\begin{equation*}
\zeta_{b}(z)= \tr \left( b \AD ^{-2z} \right)
\end{equation*}
are well defined and analytic for $\text{Re} \, z >  p/2$ and all $b $ in the algebra $\B$ generated by the elements
\begin{equation*}
\delta^k(\pi(a)), \quad \delta^k(\lcom \D, \pi(a) \rcom ) \quad \text{for } a \in \A \, \text{ and } \, k \geq 0.
\end{equation*}
In order to apply the local index theorem \cite{connes_index} of A. Connes and H. Moscovici we need even more. More precisely, we have to introduce the notion of the \textit{dimension spectrum} of a spectral triple. This is the set $\Sigma \subset \CH$ of singularities of functions
\begin{equation*}
 \zeta_{b}(z)= \tr \left( b \AD ^{-2z} \right) \quad \text{Re} \, z >  p/2 \quad b \in \B \, .
\end{equation*}
We assume the following:

\begin{3_ax}
$\Sigma$ is a discrete subset of $\CH$. Therefore, for any element $b$ of the algebra $\B$ the functions 
\begin{equation*}
\zeta_{b}(z) = \tr(b \mid \mathcal{D} \mid ^{-2z}) 
\end{equation*}
extend holomorphically to $\CH \setminus \Sigma$.
\end{3_ax}

By means of the third condition the functions $ \zeta_{b} $ are meromorphic for any $b\in \B$. In the commutative case of a $p$-dimensional manifold $M$ the dimension spectrum $\Sigma$ consists of all positive integers less or equal $p$ and is simple, i. e. the poles of $\zeta_{b}$ are all simple. All spectral triples we are going to work with have simple dimension spectrum. \\
Let us introduce pseudo differential calculus (see \cite{connes_index} for a full discussion). Let $\text{dom}(\delta)$ denote the domain of the map $\delta$ which consists of all $a \in \B(\H)$ such that $\delta(a)$ is densely defined on $\H$ and can be extended to a bounded operator. We shall define the order of operators by the following filtration: For $ r \in \RH $ set
\begin{equation*}
 OP^{0} = \bigcap_{n \geq 0} \text{dom} (\delta ^{n} ) , \quad OP^{r} = \mid \D \mid ^{r} OP^{0}.
\end{equation*}
Due to regularity we obtain $\B \subset OP^{0} $. Additionally, for every $P $ in $OP^{r}$ the operator $ \AD ^{-r} P $ is densely defined and has a bounded extension. \\
Let $\nabla$ be the derivation $\nabla(T) = \lcom \D^2 , T \rcom $ for an operator $T$ on $\H$. Consider the algebra $ \mathfrak{D}$ generated by
\begin{equation*}
\nabla^k(T), \quad T \in \A \,   \text{ or } \,  \lcom \D , \A \rcom 
\end{equation*} 
for $k \geq 0$. The algebra $\mathfrak{D}$ is the analogue of the algebra of differential operators. By corollary B.1 of \cite{connes_index} we obtain $\nabla^k(T) \in OP^k$ for any $T$ in $\A$ or $\lcom \D, \A \rcom$. \\ 
We introduce the notation
\begin{equation*}
\tau_{k} \left( b \mid \D \mid ^{y} \right)  = \res_{z=0} \, z^{k} \, \tr \left( b \mid \D \mid^{y-2z} \right),
\end{equation*}
for $ b \in \mathfrak{D} $, $ y \in \Sigma$ and an integer $ k \geq 0 $. These residues are independent of the choice of the definition of $\AD ^{-1}$ on the kernel of $\D$ because the kernel is finite dimensional. A change of the definition of $\AD^{-1}$ on the kernel of $\AD$ results in a change by a trace class operator. \\

Now we are ready to state the local index theorem proven in \cite{connes_index} for spectral triples which fulfil the three conditions above. First, let us recall the definition of the index. Let $ u \in \A$ be a unitary. Then, by definition of a spectral triple, the commutator $\lcom \D, \pi(u) \rcom$ is bounded. This forces the compression $P \pi(u) P$ of $\pi(u)$ to be a Fredholm operator (if $F = \sgn \D $ then $P=(F+1)/2$). The index is defined by 
\begin{equation*}
\ind ( P \pi(u) P ) = \text{dim ker} \, P \pi(u)P - \text{dim ker} \, P \pi(u)^\ast P.
\end{equation*}
The index map $u \mapsto \ind(P \pi(u)P)$ defines a homomorphism from the $K_1$-group of the pre-$C^{\ast}$-algebra $\A$ to $\ZH$, see \cite{connes}. By the local index theorem this index can be calculated as a sum of residues $\tau_{k}(b \AD^{y})$, for various $b \in \mathfrak{D}$. Since we are only interested in the three dimensional case we state the local index theorem (Corollary II.1 of \cite{connes_index}) only for this case.

\begin{thm} \label{local_index}
Let $(\A, \H, \D)$ be a three dimensional spectral triple fulfilling the three conditions above and let $u \in \A$ be a unitary. Then 
\begin{equation*}
\ind \, P \pi(u) P = \phi_{1}( u^{*}du) - \phi_{3}(u^{*}dudu^{*}du),
\end{equation*}
where
\begin{align*} 
 \phi_{3}(a^0da^1da^2da^3) = & \quad \dfrac{1}{12} \, \tau_0 \left( \pi(a^0) \lcom \D, \pi(a^1) \rcom \lcom \D, \pi(a^2) \rcom \lcom \D, \pi(a_{3}) \rcom \AD^{-3}  \right) \\
&- \dfrac{1}{6} \, \tau_1 \left( \pi(a^0) \lcom \D, \pi(a^1) \rcom \lcom \D, \pi(a^2) \rcom\lcom \D, \pi(a^3) \rcom \AD^{-3}  \right),
\end{align*}
and
\begin{align*}
 \phi_{1} (a^{0}da^{1})  =& \quad \tau_0 \left( \pi(a^0) \lcom \D, \pi(a^1) \rcom \AD^{-1} \right) \\ 
			&- \dfrac{1}{4} \, \tau_0 \left( \pi(a^0) \nabla \left( \lcom \D, \pi(a^1) \rcom \right) \AD^{-1} \right) \\
			&- \dfrac{1}{2} \, \tau_1 \left( \pi(a^0) \nabla \left( \lcom \D, \pi(a^1)\rcom \right) \AD^{-3} \right) \\
			&+ \dfrac{1}{8} \, \tau_0 \left( \pi(a^0) \nabla^{2} \left( \lcom \D, \pi(a^1) \rcom \right) \AD^{-5} \right) \\
			&+ \dfrac{1}{3} \, \tau_1 \left( \pi(a^0) \nabla^{2}\left( \lcom \D, \pi(a^1) \rcom \right) \AD^{-5} \right) \\
			&+ \dfrac{1}{12}\, \tau_2 \left( \pi(a^0) \nabla^{2} \left( \lcom \D, \pi(a^1) \rcom \right) \AD^{-5} \right)
\end{align*}
for $a^0,a^1,a^2,a^3 \in \A$. In addition, the pair $(\phi_3, \phi_1)$ defines a cyclic cocycle in Connes' $(B,b)$-complex.
\end{thm}

\medskip

\subsection{Chern-Simons action on noncommutative spaces} 

If $( \A, \H, \D ) $ is a spectral triple, then $( M_{N}( \A ), \H \otimes \CH^{N}, \D \otimes I_{N} ) $ is a spectral triple over the matrix algebra $M_N(\A)$ of $\A$, for a positive integer $N \in \NH$, satisfying all requirements imposed to $( \A, \H, \D )$. In the sequel, we shall work with the second one, whose differential forms $\Omega^1(M_{N}(\A)) = M_{N} \left( \Omega^1(\A) \right) $ are matrices with values in $\Omega^1(\A)$. The Dirac operator of the extended triple will be also denoted by $\D$ instead of $\D \otimes I_{N}$. A matrix 1-form $A = \left( a_{ij} db_{ij} \right) _{1 \leq i,j \leq N}$ is called hermitian if $A = A ^{\ast}$ with
\begin{equation*}
\left( \left( a_{ij} db_{ij} \right) _{i,j \leq N } \right) ^{\ast} = \left( \left( a_{ji}db_{ji} \right)^{\ast} \right) _{1 \leq i,j \leq N } \, ,
\end{equation*} 
where the $\ast$-operator on $\Omega(\A)$ is defined as follows: 
\begin{equation*}
 (a^0da^1 \cdots da^n )^{\ast} = (-1)^n d a^{n \ast} \cdots da^{1 \ast} a^{0 \ast} \, 
\end{equation*} 
for $a^0, \ldots, a^n \in \A$.

\begin{defn} \label{CS}
 Let $( \A, \H, \D ) $ be a spectral triple satisfying the conditions of section 2.1, $N \in \NH$, and $A \in M_{N} \left( \Omega^1(\A) \right)$ a hermitian matrix of 1-forms. We define the Chern-Simons action as \[S_{CS}(A) = 6 \pi k \phi_{3}\left( AdA +  \frac{2}{3} A ^{3} \right) - 2 \pi k \phi_{1}(A)\] for an integer $k$, with the cyclic cocycle $(\phi_3,\phi_1)$ of the local index theorem \ref{local_index}.
\end{defn} 

Let us introduce the notion of a gauge transformation. The gauge group of $\left( M_{N}\left( \A \right), \H \otimes \CH^{N}, \D \otimes I_{N} \right) $ is the group of unitary elements of $ M_{N}( \A )$. An element $u \in  M_{N}( \A )$ is called unitary if $u^{\ast}u=uu^{\ast}=1$. A gauge transformation of a 1-form $A \in M_{N} ( \Omega^1(\A) )$ by a unitary $u \in M_{N}( \A )$ is given by the transformation law 
\begin{equation*}
A \mapsto uAu^{\ast} + u du^{\ast}. 
\end{equation*}
Note if $A \in M_N(\Omega^1(\A))$ is a hermitian matrix of $1$-forms, then this holds true for $A^u$ too. This leads us to them main result of \cite{pfante_chern_simons}(see theorem 3.1).

\begin{thm} 
 Let $( \A, \H, \D ) $ be a spectral triple satisfying the conditions of section 2.1, $N \in \NH$, and $A \in M_{N} ( \Omega^1(\A) )$ a hermitian matrix of 1-forms. Then under the gauge transformation of $A$ by a unitary $u \in M_{N}(\A)$, the Chern-Simons action becomes 
\begin{equation*}
 S_{CS}(uAu^{\ast} + u du^{\ast}) = S_{CS}(A) + 2 \pi k \,  \ind(P \pi(u) P), 
\end{equation*}
where $P = (1+F)/2$ and $F = \sgn  \D$ is the sign of $\D$. 
\end{thm}

\section{Chern-Simons action for the noncommutative $3$-torus $\torus$.}

In this section we compute the Chern-Simons action of definition \ref{CS} explicitly for the noncommutative $3$-torus $\torus$ and its spectral triple constructed in \cite{iochum}. We start with an outline of the results in \cite{iochum}.

\medskip

\subsection{The spectral triple of $\torus$}

Let $\torus$ be the smooth noncommutative $3$-torus associated to a non-zero, skew-symmetric deformation matrix $\Th \in M_{3}(\RH) $. This means that $\torus$ is the $\ast$-algebra generated by $3$ unitaries $u_i$, $i = 1, 2,3$, which are subject to the relations
\begin{equation*}
u_i u_j = e ^{i \Theta_{ij}} u_j u_i ,
\end{equation*}
and with Schwartz coefficients: an element $a \in \torus$ can be written as $a = \sum _{k \in \ZH^3} a_k U_k$, where $ \left\lbrace  a_k \right\rbrace \in \S(\ZH^3)$ with the Weyl elements defined by 
\begin{equation*}
U_k := e^{- \frac{i}{2} k. \chi k} u_1 ^{k_1} u_2 ^{k_2} u_3 ^{k_3} \quad \fl k = ( k_1, k_2, k_3) \in \ZH^3, 
\end{equation*}
and the above relation reads
\begin{equation*}
U_k U_q = e^{- \frac{i}{2} k. \Theta q} U_{k+q} \quad U_{k}U_{q} = e^{-i k. \Theta q} U_{q} U_{k}
\end{equation*}
where $\chi$ is the matrix restriction of $ \Th $ to its upper triangular part. Thus the unitary operators $U_k$ satisfy $ U_k ^{\ast} = U_{-k}$ and
\begin{equation*}
\lcom U_k , U_l \rcom = -2i \sin \left( \frac{1}{2} k. \Theta l \right) U_{k+l}.
\end{equation*} 
Let $\tau$ be the trace on $\torus$ defined by 
\begin{equation*} 
\tau \left( \sum_{k \in \mathbb{Z}^3} a_k U_k \right)= a_0  
\end{equation*} 
and $\H_{\tau}$ be the GNS Hilbert space obtained by completion of $\torus$ with respect to the norm induced by the scalar product $ \braket{a}{b}  := \tau(a^{\ast}b)$. On 
\begin{equation} \label{GNS}
\H_{\tau} = \left\lbrace \sum_{k \in \ZH^3} a_k U_k : \,  \left\lbrace a_k \right\rbrace _k  \in l^2(\ZH^3)\right\rbrace \, ,
\end{equation}
we consider the left regular representations of $\torus$ by bounded operators which we denote by $L(.)$.\\
Let also $\d_{\mu}$, $\mu \in \set{1,2,3}$, be the 3 (pairwise commuting) canonical derivations, defined by
\begin{equation*}
\delta_{\mu}(U_k) := i k_{\mu} U_k \, .
\end{equation*} 
We need to fix notations: let $\A_\Th := \torus$ acting on $\H := \H_{\tau} \ot \CH^2$ the square integrable sections of the trivial spin bundle over $\TH^3.$ \\
Each element of $\A_\Th$ is represented on $\H$ as $L(a) \ot I_{2}$, where $L$ is the left multiplication. The Dirac operator is given by
\begin{equation*}
\D := -i \d_{\mu} \ot \g^{\mu},
\end{equation*} 
where we use summation notation and hermitian Dirac matrices
\begin{equation*}
\g^1 = \left(\begin{array}{cc}
0 & 1 \\ 
1 & 0
\end{array} \right) , \, \g^2 = \left( \begin{array}{cc}
0 & -i \\ 
i & 0
\end{array} \right), \,  \g^3 = \left( \begin{array}{cc}
1 & 0 \\ 
0 & -1 
\end{array} \right)  
\end{equation*}
By the definitions above one deduces quite easily that $(\A_\Th, \H, \D)$ is a spectral triple which fulfils the regularity and dimension condition of section 2.1. In order to achieve the dimension spectrum condition we need to impose a so called \textit{diophantine} condition on the skew-symmetric matrix $\Th$. See definition 2.4 in \cite{iochum}

\begin{defn}
(i) Let $\d > 0$. A vector $a \in \RH^n$ is said to be $\d$-diophantine if there exists a $c > 0$ such that $\mid q \cdot a - m \mid \geq  c q^{ - \d }$ for all $ q \in \ZH^n \setminus \set{0} $ and for all $m \in \ZH$. We denote by $\mathcal{BV}(\delta)$ the set of $\d$-diophantine vectors and $ \mathcal{BV}:= \bigcup _{\delta > 0} \mathcal{BV}(\delta) $ the set of diophantine vectors. \\
(ii) A real $n \times n$ matrix $\Th \in M_n (\RH)$ is said to be diophantine if there exists a vector $u \in \ZH^{n}$ such that $( \Th(u))^T$ is a diophantine vector of $\RH^n$.
\end{defn}

\begin{rem}
A classical result from Diophantine approximation asserts that for all $\d > n$, the Lebesgue measure of $\RH^n \setminus \mathcal{BV}(\delta)$ is zero (i.e., almost every element of $\RH^n$ is $\d$-diophantine).\\
Let $ \Th \in M_n(\RH). $ If its row of index $i$ is a diophantine vector of $\RH^n$ then $(\Th(e_i))^T \in \mathcal{BV}$ and thus $\Th$ is a diophantine matrix. It follows that almost any matrix of $M_n(\RH) \approx \RH^{n^2}$ is diophantine.
\end{rem}

Proposition 5.4 of \cite{iochum} yields:

\begin{thm} 
If $\frac{1}{2 \pi} \Th$ is diophantine, then the dimension spectrum of $(\A_\Th, \H, \D)$ is equal to the set $ \set{3-k : \, k \in \NH_0 }$ and all these poles are simple.
\end{thm}

In the sequel we always work with a deformation matrix $\Th$ such that $\frac{1}{2\pi} \Th$ is diophantine without mentioning it explicitly. Additionally, in order to keep the notation as simple as possible, we identify $a$ and $L(a)$ for all $a \in \A_\Th$. \\

Finally, we study the local index formula of $(\A_\Th, \H, \D)$. Due to the simplicity of the dimension spectrum the formulas for the cyclic cocycle $(\phi_3 , \phi_1 )$ simplify as follows
\begin{equation} \label{phi_3}
\phi_{3}(a^0,a^1,a^2,a^3) = \dfrac{1}{12} \, \tau_0 \left( a^0 \lcom \D, a^1 \rcom \lcom \D, a^2 \rcom \lcom \D, a^3 \rcom \AD^{-3}  \right), 
\end{equation}
and
\begin{align*}
\phi_{1} (a^{0}, a^{1})  =& \tau_0 \left(a^0 \lcom \D, a^1 \rcom \AD^{-1} \right) \\ 
	&- \dfrac{1}{4} \, \tau_0 \left(a^0 \nabla \left( \lcom \D, a^1 \rcom \right) \AD^{-3} \right) \\
	&+ \dfrac{1}{8} \, \tau_0 \left( a^0 \nabla^{2} \left( \lcom \D, a^1 \rcom \right) \AD^{-5} \right)
\end{align*}
for $a^0,a^1,a^2,a^3 \in \A$ and $\nabla(T)= \lcom \D^{2}, T \rcom $ for any $T$ in $ \A_\Th$ or $\lcom \D, \A_\Th \rcom $.

\begin{lem} \label{phi_1}
The cochain $\phi_1$ of the local index theorem vanishes for the triple $(\A_\Th, \H, \D)$. I. e.
\begin{equation*}
\phi_1(a^0da^1)= 0 \quad \fl a^0,a^1 \in \A_\Th
\end{equation*}
\end{lem}

\begin{proof}
For any element $a = \sum_{k \in \ZH^3} a_k U_k$ we obtain 
\begin{equation*}
\lcom \D, a \rcom = \sum_{k \in \ZH^3} a_k \lcom \D, U_k \rcom = \sum_{k \in \ZH^3} a_k k_{\mu} U_k \ot \g^{\mu} = -i \d_{\mu}a \ot \g^{\mu}
\end{equation*}
The operator $\D^2$ has the form
\begin{equation*}
\D^2 = \left( -i \d_{\mu} \ot \g^{\mu} \right)^2 = - \delta_{\l}\d_{\mu} \ot \g^{\l} \g^{\mu} = - \sum_{\mu} \d_{\mu}^2 \ot I_2,
\end{equation*}
where the last identity follows by the fact that $\g^\l \g^\mu = - \g^\mu \g^\l$ for $\mu \neq \l$ and $\d_\mu \d_\nu = \d_\nu \d_\mu$. By an analogous calculation as above we obtain 
\begin{equation*}
\lcom \D^2, a \rcom = - \sum_{\mu}  \d_{\mu}^2 a + 2 \d_\mu a \d_\mu \ot I_2
\end{equation*}
Thus the individual terms of $\phi_1$ have the form
\begin{align*}
a^0 \lcom \D, a^1 \rcom =& -i a^0 \d_{\l} a^1 \ot \g^{\l} \\
a^0 \nabla \lcom \D, a^1 \rcom =& \quad i a^0 \sum_{\mu} \d_\mu^2 \d_\l a^1 \ot \g^\l + 2 \d_\mu \d_\l a^1 \d_\mu \ot \g^\l \\
a^0 \nabla^2 \lcom \D, a^1 \rcom =& \quad i a^0 \sum_{\mu} \nabla \left( \d_{\mu}^2 \d_{\l} a^1 \ot \g^\l + 2 \d_\mu \d_\l a^1 \d_\mu \ot \g^\l \right) \\
=&  -i a^0 \sum_{\mu, \nu} \d_\nu^2 \d_\mu^2 \d_\l a^1 \ot \g^\l + 2 \d_\nu^2 \d_\mu \d_\l a^1 \d_\mu \ot \g^\l \\
& \qquad+ 2 \d_\nu \d_\mu^2 \d_\l a^1 \d_\nu \ot \g^\l  + 4 \d_\nu \d_\mu \d_\l a^1 \d_\nu \d_\mu \ot \g^\l
\end{align*}
which are all operators on $\H$ with zero trace because the trace of the Dirac matrices $\g^{\mu}$ is zero. This yields $\phi_1 = 0$. 
\end{proof}

For the extended triple $( M_N(\A_{\Th}), \H \ot \CH^N, \D \ot I_N )$, which is the triple we are working with in the sequel, the derived results hold also true.

\medskip

\subsection{Chern-Simons action}

We extend the left-multiplication $L \ot I_2$ of the previous subsection to a representation of $\Omega(\A_\Th)$ by 
\begin{equation*}
a^0da^1 \cdots da^n \mapsto a^0 \lcom \D, a^1 \rcom \cdots \lcom \D, a^n \rcom \,
\end{equation*}
which we denote also with $L$. Although this map defines a $\ast$-representation of $\Omega(\A_\Th)$ as a $\ast$-algebra (see \cite{connes} chapter VI, proposition 4), it fails to be a representation of the differential structure, i. e. $L(da) = \lcom \D , L(a) \rcom$ does not hold true in general.\\
Let $A= \sum_i a_i db_i$ be a hermitian $N \times N$-matrix of $1$-forms, and $A_\l = \sum_i a_i \d_\l b_i$. In addition, let denote the extension of $L$ on $M_N(\Omega(\A_\Th)) \cong M_N(\CH) \ot \Omega(\A_\Th)$ also with $L$. Then we obtain
\begin{equation*}
L(A) = -i A_\l \ot \g^\l \, ,
\end{equation*}
with $A_\l^\ast = - A_\l$ because the $1$-form $A$ is assumed to be hermitian. In order to compute the Chern-Simons action of definition \ref{CS} we need to know $L(dA)$ too. 
\begin{align*}
L(dA) =& \sum_i \lcom \D, a_i \rcom \lcom \D, b_i \rcom \\
=& - \sum_i \d_\l a_i \d_\mu b_i \ot \g^\l \g^\mu  \\
=& - \sum_i \d_\l (a_i \d_\mu b_i ) \ot \g^\l \g^\mu - a_i \d_\l \d_\mu b_i \ot \g^\l \g^\mu \\
=& -  \d_\l A_\mu \ot \g^\l \g^\mu - \sum_\l \sum_i a_i \d_\l^2 b_i \ot I_2 \, ,
\end{align*}

where the last row follows from the previous one by the relations $\d_\l \d_\mu = \d_\mu \d_\l$ and $\g^\l \g^\mu = - \g^\mu \g^\l$  if $\l \neq \mu$. Summarizing these insights and combining them with lemma \ref{phi_1} and equation \eqref{phi_3} yields
\begin{align*}
& S_{CS}(A) = 6 \pi k \phi_3 \left(AdA + \dfrac{2}{3} A^3 \right) \\
&= \dfrac{\pi k}{2} \tau_0 \left( \left( L(A) L(dA) + \dfrac{2}{3} L(A)^3 \right) \AD^{-3} \right) \\
&= \dfrac{\pi k}{2} \res_{z=0} \tr_{\H \ot \CH^N} \left( i \left(  A_\l \d_\mu A_\nu +  \dfrac{2}{3} A_\l A_\mu A_\nu \right)  \ot \g^\l \g^\mu \g^\nu \AD^{-3-2z} \right) \\
&+ \dfrac{\pi k}{2} \res_{z=0} \tr_{\H \ot \CH^N} \left( \left( i A_\l \left(\sum_\mu \sum_i a_i \d_\mu^2 b_i \right) \ot \g^\l \right) \AD^{-3-2z} \right) \\
&= - \pi k \e^{\l \mu \nu} \res_{z=0} \tr_{\H_\tau \ot \CH^N} \left( \left( A_\l \d_\mu A_\nu  +  \dfrac{2}{3} A_\l A_\mu A_\nu  \right) \left( - \d_1^2 - \d_2^2 - \d_3^2 \right)^{-3/2-z} \right) \, ,
\end{align*}

where the last row follows by the previous one by the following facts: $\H = \H_\tau \ot \I_2$, with the Hilbert space $\H_\tau$ of the GNS construction \eqref{GNS}, $\tr_{\CH^2}(\g^\l \g^\mu \g^\nu) = 2i \e^{\l \mu \nu}$, $\tr_{\CH^2}(\g^\l) = 0 $ and $\D^2 = - \d_1^2 -\d_2^2 - \d_3^2 \ot I_2$.\\
By definition of $\A_\Th$ or its extended version $M_N (\A_\Th)$ respectively, there is a rapidly decreasing sequence $\left( a^\l_{k} \right)_{k \in \ZH^3}$ of complex $N \times N$-matrices $ a^\l_k = \left( a^{\l, ij}_k \right) $, the Fourier coefficients, with 
\begin{equation} \label{fourier} 
A_{\l} = \sum_{k \in \ZH^3} a^\l_{k}U_{k}  
\end{equation} 
for $\l = 1,2,3$. Due to $A_\l^\ast = - A_\l$ we have $\left( a^\l_k \right)^\ast = - a_{-k}^{\l}$ for all $k \in \ZH^3$.

\begin{rem} \label{non_self} In anticipation of the Feynman rules of section four we exclude self-interactions in our field theory. This means that loops which contain vertices whose legs are interconnected do not contribute anything to the loop expansion series of theorem \ref{partition}. \\
This can be achieved by the setting $a^\l_0 = 0$, for $\l = 1,2,3$, in equation \eqref{fourier}.
\end{rem}

Using the Fourier decomposition \eqref{fourier} we can proceed with our computation and obtain
\begin{align*}
A_\l \d_\mu A_\nu &= \sum_{\begin{tiny}
\begin{array}{c}
q \in \ZH^3 \setminus \set{0}, n \in \ZH^3 \\ 
n-q \neq 0
\end{array} \end{tiny}} a^\l_{n-q} (iq_\mu) a^\nu_q \exp \left( \frac{-i}{2} n. \Th q \right)  U_n \\
A_\l A_\mu A_\nu &= \sum_{\begin{tiny}
\begin{array}{c}
r \in \ZH^3 \setminus \set{0}, n,q \in \ZH^3 \\ 
n - q,q-r \neq 0
\end{array} \end{tiny}} a^\l_{n-q} a^\mu_{q-r} a^\nu_{r} \exp \left( \frac{-i}{2} q. \Th r \right) \exp \left( \frac{-i}{2} n. \Th q \right) U_n
\end{align*}

We fix an orthonormal basis $(E_l \ot e_j)_{l,j}$ of the Hilbert space $\H_\tau \ot \CH^N $, where $ E_l := U_l$, for $l \in \ZH^3$, and $(e_j)_{j = 1, \ldots, N}$ is the canonical basis of $\CH^{N}$. $ E_l \ot e_j $ is an eigenvector of $- \d_1^2- \d_2^2 - \d_3^2$ with eigenvalue $ \mid l \mid ^2 := l_1^2 + l_2^2 + l_3^2$. We compute
\begin{align*}
&- \pi k \, \e^{\l \mu \nu} \res_{z=0} \, \tr_{\H_\tau \ot \CH} \left( A_\l \d_\mu A_\nu (-\d_1 -\d_2 -\d_3)^{-3/2-z} \right) \\
=& - \pi k \, \e^{\l \mu \nu} \res_{z=0} \sum_j \sum_l \braket{E_l \ot e_j}{A_\l \d_\mu A_\nu (-\d_1 -\d_2 -\d_3)^{-3/2-z} E_l \ot e_j} \\
=&- \pi k \, \e^{\l \mu \nu} \res_{z=0} \sum_{j} \sum_l \mid l \mid^{-3-2z} \times \\
&\left\langle E_l \ot e_j , \sum_{\begin{tiny}
\begin{array}{c}
q \in \ZH^3 \setminus \set{0}, n \in \ZH^3 \\ 
n-q \neq 0
\end{array} \end{tiny}}  a^\l_{n-q} (iq_\mu) a^\nu_q \exp \left( \frac{-i}{2} n. \Th q \right)  U_{n} (E_l \ot e_j) \right\rangle \\
=&- \pi k \,  \e^{\l \mu \nu} \res_{z=0} \sum_l \mid l \mid^{-3-2z} \sum_{q \in \ZH^3 \setminus \set{0}} \tr_{\CH^{N}} \left( a^\l_{-q} (iq_\mu) a^\nu_{q} \right) \\
=& \quad 2 \pi^2  k \, \e^{\l \mu \nu} \sum_{q \in \ZH^3 \setminus \set{0}}  \ov{a^{\l, ij}_{q}} (iq_\mu) a^{\nu, ij}_q \\
=& \quad 2 \pi^2 i k \sum_{q \in \ZH^3 \setminus \set{0}}  \left( \ov{a_{q}^{1,ij}}, \ov{a_{q}^{2,ij}}, \ov{a_{q}^{3,ij}} \right)   \left( \begin{array}{ccc}
0 & -q_3 & q_2 \\ 
 q_3 & 0 & -q_1 \\ 
-q_2 & q_1 & 0
\end{array}  \right) \left( \begin{array}{c}
a_{q}^{1,ij} \\ 
a_{q}^{2,ij} \\ 
a_{q}^{3,ij}
\end{array} \right)
\end{align*}

where the last line follows by the fact that $\res_{z=0} \sum_l \mid l \mid^{-3-2z} = 2 \pi$, due to theorem 2.1 \cite{iochum}. The calculations for the cubic term are analogous and we obtain
\begin{align*}
& - \pi  k \, \e^{\l \mu \nu} \res_{z=0} \, \tr_{\H_{\tau} \ot \CH^{N}} \left( A_\l A_\mu A_\nu (-\d_1-\d_2-\d_3)^{-3/2-z} \right) \\
=&- 2 \pi^2  k \, \epsilon^{\l \mu \nu} \sum_{\begin{tiny}
\begin{array}{c}
q,r \in \ZH^3 \setminus \set{0} \\ 
q-r \neq 0
\end{array} \end{tiny}}  \exp \left( \frac{-i}{2} q. \Th r \right) a^{\l, ij}_{-q} a^{\mu,jm}_{q-r} a^{\nu,mi}_{r} 
\end{align*}

Summarizing these results yields the final expression for the Chern-Simons action on the noncommutative 3-torus.

\begin{thm} \label{chern_simons_ac_3_torus}
Let $A \in M_N(\Omega^{1}(\A_{\Th}))$ be a hermitian matrix of 1-forms which fulfils the condition of remark \ref{non_self}, and 
\begin{align*}
 L: M_N (\Omega(\A_\Th)) & \ra \B \left( \H \ot \CH^N \right) \\
 a^0da^1 \cdots da^n & \mapsto a^0 \lcom \D, a^1 \rcom \cdots \lcom \D, a^n \rcom 
\end{align*}
an extension of the left multiplication $L \ot I_2 : \A_\Th \ra \B(\H)$. Then, there are skew hermitian elements
\begin{equation*}
A_\l = \sum_{k \in \ZH^3 \setminus \set{0}} a^\l_k U_k \, ,
\end{equation*}
with a rapidly decreasing series of $N \times N$-matrices $a^\l_k=  \left( a^{\l, ij}_k \right) $, fulfilling $a_k^{\l \ast} = - a^\l_{-k}$ for all $k \in \ZH \setminus \set{0}$, such that
\begin{equation*}
L(A) = -i A_\l \ot \g^\l \, 
\end{equation*}
and the Chern-Simons action of definition \ref{CS} is given by the matrix coefficients as follows: 
\begin{align*} 
S_{CS}(A)=& \quad 2 \pi^2 i k \sum_{q \in \ZH^3 \setminus \set{0}} \left( \ov{a_{q}^{1,ij}}, \ov{a_{q}^{2,ij}}, \ov{a_{q}^{3,ij}} \right)   \left( \begin{array}{ccc}
0 & -q_3 & q_2 \\ 
 q_3 & 0 & -q_1 \\ 
-q_2 & q_1 & 0
\end{array}  \right) \left( \begin{array}{c}
a_{q}^{1,ij} \\ 
a_{q}^{2,ij} \\ 
a_{q}^{3,ij}
\end{array} \right) \\
&- 2 \pi^2  k \, \e^{\l \mu \nu} \sum_{\begin{tiny}
\begin{array}{c}
q,r \in \ZH^3 \setminus \set{0} \\ 
q-r \neq 0
\end{array} \end{tiny}}  \exp \left( \frac{-i}{2} q. \Th r \right) a^{\l, ij}_{-q} a^{\mu,jm}_{q-r} a^{\nu,mi}_{r} 
\end{align*}
\end{thm}

\section{path integral}

We introduce and compute a noncommutative replacement 
\begin{equation} \label{path_integral}
Z(k) = \int D  \lcom A \rcom \, \exp \left(i S_{CS}(A) \right) \, , 
\end{equation}
of the path integral \eqref{eq_1.2}, where the Chern-Simons functional is integrated over all equivalence classes of hermitian matrices of $1$-forms $A \in M_N(\Omega(\A_\Th))$ modulo gauge transformations $ A \mapsto u A u^\ast + udu^\ast $. We start with an interpretation of the measure $D \lcom A \rcom$.\\
From theorem \ref{chern_simons_ac_3_torus} one can read off that the Chern-Simons action depends only on the variables $a^{\l, ij}_k$, with $\l =1,2,3$, $i,j = 1, \ldots, N$ and $k \in \ZH^3 \setminus \set{0}$. Thus the infinite product of Lebesgue measures
\begin{equation*}
D A = \prod_{i,j=1}^N \prod _{\l = 1}^3 \prod_{q\in \ZH^3 \setminus \set{0}} da^{\l,ij}_{q}
\end{equation*}
seems to be the natural choice for our noncommutative path integral. If we used the measure $ D A$, we would integrate over all possible hermitian matrices of $1$-forms $A$, not only over the equivalence classes with respect to gauge transformation. In order to avoid this we have to break the gauge symmetry. 

\medskip

\subsection{Gauge-breaking}

We choose the Lorentzian gauge fixing, i. e. $ h(A)  - \omega = 0 $, with $h(A) = -i \d_\l A_\l $ and $\omega \in M_N(\B(\H))$. If one applies the standard techniques of Faddeev and Popov (thoroughly discussed in the lecture notes \cite{jesper} of A. Ambj\o rn and J. L. Petersen) by introducing ghost-fields $c,\oc$ and the Lautrup-Nakanishi field $B$, one achieves the following form of \eqref{path_integral}
\begin{equation} \label{gauge_breaking}
\int DB D \oc Dc DA \, \exp \left(i S_{CS}(A)+ \pi^2 \ncint \left(i (\d_\mu A_\mu)B - i \oc \d_\mu D_\mu c \right)\AD^{-3} \right) \, ,
\end{equation}
with the covariant derivative $D_\mu = \d_\mu + \lcom A_\mu , \cdot \rcom$. The additional ingredients of the formula are defined as follows:
\begin{itemize}
\item \begin{equation*}
\ncint T := \res_{z=0} \tr \left( T \AD^{-z} \right)
\end{equation*}
denotes the Wodzicki residue of an operator $T$. The additional factor $\pi^2$ in front of $\ncint$ is a normalization factor in order to make the residue in the commutative case, i. e. $\Th=0$, identical with the ordinary volume integral on the $3$-dimensional torus $\mathbb{T}^3$.
\item The ghost-fields $c$ and $ \oc $, with their Fourier decompositions
\begin{equation*}
\oc = \sum_{q \in \ZH^3 \setminus \set{0}} \oc_{-q} U_q \qquad c = \sum_{q \in \ZH^3 \setminus \set{0}} c^{T}_q U_q \, ,
\end{equation*}
 where $\oc_q$ and $c_q $ are $N \times N$-matrices of anticommuting Grasmann numbers for all $q \in \ZH^3 \setminus \set{0}$. $c^T_q$ denotes the transpose of $c_q$.
\item The Lautrup-Nakanishi field $B$ is a hermitian element of $M_N(\A_\Th)$ with Fourier decomposition
\begin{equation*}
B =  \sum_{q \in \ZH^3 \setminus \set{0}} i B_q U_q \, ,
\end{equation*}
where $B_q$ are $N \times N$-matrices of rapid decay fulfilling $B_q^\ast = - B_q$ for all $q \in \ZH^3 \setminus \set{0}$.
\end{itemize}

Next we introduce external source fields $J_{q}^{\mu, ij}, \xi^{ij} , \xi^{\ast ij} $ in order to create a generating functional for \eqref{gauge_breaking}. We think of $J_{q}^{ \mu, ij}$ being ordinary complex variables, with $\ov{J_q^{\mu, ij}}= J_{-q}^{\mu, ji} $ for all $q \in \ZH^3 \setminus \set{0}$, instead of $\xi^{ij}$ and $\xi^{\ast ij}$ which are Grassmann numbers. Adding external fields turns the path integral \eqref{gauge_breaking} into a Gaussian integral which can be computed as in the finite dimensional case (see G. B. Folland's book \cite{folland} for further details). The Gaussian integral decouples in the indices $i,j = 1, \ldots, N$ and we can be assemble them. We obtain the following result.

\begin{thm} \label{partition}
The partition function
\begin{equation*}
Z(K) = \int  D \lcom A \rcom e^{i S_{CS}(A)},
\end{equation*} 
where we integrate over all equivalence classes $\lcom A \rcom $ of hermitian $N \times N$-matrices  $A \in M_N(\Omega^1(\A_\Th))$ of $1$-forms modulo gauge transformations, is given by the formula:
\begin{align*}
& \quad  \sum_{n = 0}^{\infty} \dfrac{1}{n!} \prod_{g=1}^{n} -2 \pi^2 i  k \, N^{3/2}  \e^{\l_g \mu_g \nu_g} \sum_{\begin{tiny}
\begin{array}{c}
q^g,r^g \in \ZH^3 \setminus \set{0} \\ 
q^g-r^g \neq 0
\end{array} \end{tiny}} \sin \left( \frac{1}{2} q^g. \Th r^g \right) \dfrac{\partial}{2i \partial J_{q^g}^{\l_g}} \dfrac{\partial}{2i \partial J_{r^g-q^g}^{\mu_g}}\dfrac{\partial}{2i \partial J_{-r^g}^{\nu_g}}\\
&\times \sum_{m=0}^{\infty} \dfrac{1}{m!} \prod_{h=1}^{m}  -16 \pi^3 i \, N^{3/2} \sum_{\begin{tiny}
\begin{array}{c}
q^h,r^h \in \ZH^3 \setminus \set{0} \\ 
q^h-r^h \neq 0
\end{array} \end{tiny}} \sin \left( \frac{1}{2} q^h. \Th r^h \right)  q^h_\mu \dfrac{\partial}{i \partial \xi_{q^h}}  \dfrac{\partial}{2i \partial J^\mu_{r^h-q^h}} \dfrac{\partial}{i \partial \xi^\ast_{r^h}}\\
& \left. \exp\left( \sum_{q \in \ZH^3 \setminus \set{0}} \dfrac{  \e^{\l \mu \nu}}{2 \pi^2 k  \mid q \mid^2 } J_{-q}^\l q_\mu J^\nu_q \right) \exp \left( \sum_{q \in \ZH^3 \setminus \set{0}} \dfrac{-i}{8 \pi^3  \mid q \mid ^2} \xi^\ast_q \xi_q \right) \right|_{ J,\xi^{\ast},\xi = 0}
\end{align*}
We call it the \textit{loop expansion series} of the partition function. Every single odd term of this expansion vanishes.
\end{thm}
  
Theorem \ref{partition} implies the following Feynman rules:
\begin{itemize} 
\item Two different kind of vertices
\begin{eqnarray*}
&\Diagram { \vertexlabel_{\l_g} \\ fd  \\
& f \vertexlabel_{\mu_g} \\
\vertexlabel_{\nu_g} fu \\
} \quad V^{A_{\l_g}A_{\mu_g}A_{\nu_g}}(q^g,r^g) &= - 2 \pi ^2 i k N^3 \, \e^{\l_g \mu_g \nu_g} \exp \left(- \frac{i}{2}q^g. \Th r^g \right) \\ && \\ && \\
&\Diagram { \vertexlabel^{c^h} \\ gd  \\
& f \vertexlabel^{\mu}  \\
\vertexlabel_{c^{\ast h}}  gu \\
} \quad V^{c^h A_{\mu} c^{\ast h}}(q^h,r^h) &= - 16 \pi^3 i N^3 \, \sin \left( \frac{1}{2} q^h. \Theta r^h \right) q^h_{\mu} \\
\end{eqnarray*}
\item Two different kind of propagators
\begin{eqnarray*}
&\feyn{\vertexlabel^{\l} !{f}{q_{\mu}}  \vertexlabel^{\nu}} \quad G^{A_{\l}A_{\nu}}(q) &= \dfrac{1}{8 \pi^2 k}\dfrac{\e^{\l \mu \nu} q_{\mu}}{\mid q \mid^2} \\
&\feyn{ \vertexlabel^{c} \,  g \, \vertexlabel^{c^{\ast}} } \quad  G^{cc^{\ast}}(q) &= \dfrac{i}{8 \pi^3} \dfrac{1}{\mid q \mid^2}
\end{eqnarray*}
\end{itemize} 
In the following subsection we compute the coefficient of $k^{-1}$ in the loop expansion series, which is given by the contributions of all $2$-loops. $2$-loops are built up by two vertices where every edge of the first vertex is linked with one edge of the other vertex by an appropriate propagator line.

\medskip

\subsection{$2$-loops}

Firstly, we investigate all loops which can be built up from these two vertices 
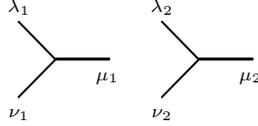
\begin{equation*}
\Diagram { \vertexlabel^{\l_1} \\ fd  \\
& f \vertexlabel_{\mu_1} \\
\vertexlabel_{\nu_1} fu \\
} \qquad \Diagram { \vertexlabel^{\l_2} \\ fd  \\
& f \vertexlabel_{\mu_2} \\
\vertexlabel_{\nu_2} fu \\
}
\end{equation*}\\
Calculating the contribution of all $2$-loops which can be constructed from these two vertices and summing up the results means nothing else than calculating the expression
\begin{align*} 
& \dfrac{1}{2} (-2 \pi^2 i k)^2 \, N^3 \, \e^{\l_1 \mu_1 \nu_1} \e^{\l_2  \mu_2 \nu_2} \sum_{\begin{tiny}
\begin{array}{c}
q^i,r^i \in \ZH^3 \setminus \set{0} \\ 
q^i-r^i \neq 0
\end{array} \end{tiny}} \sin \left( \frac{1}{2} q^1. \Th r^1 \right)  \sin \left(\frac{1}{2} q^2. \Th r^2 \right)   \\
&\dfrac{\partial}{2i \partial J_{q^1}^{\l_1}} \dfrac{\partial}{2i \partial J_{r^1-q^1}^{\mu_1}}\dfrac{\partial}{2i \partial J_{-r^1}^{\nu_1}} \dfrac{\partial}{2i \partial J_{q^2}^{\l_2}} \dfrac{\partial}{2i \partial J_{r^2-q^2}^{\mu_2}}\dfrac{\partial}{2i \partial J_{-r^2}^{\nu_2}} \label{loop} \\
&\left. \exp\left( \sum_{q \in \ZH^3 \setminus \set{0}} \dfrac{\e^{\l \mu \nu}}{2 \pi^2 k  \mid q \mid^2 } J_{-q}^{\l}q_{\mu}J^{\nu}_q \right) \right|_{ J = 0} \nonumber
\end{align*}
In order to carry out the computation we connect every leg of the first vertex with a leg of the other one. This can be done in six different ways and one obtains always 
\begin{align*}  
 L =& \dfrac{1}{2}  (2 \pi^2 i k)^2 \, N^3 \, \e^{\l_1 \mu_1 \nu_1} \e^{\l_2  \mu_2 \nu_2}  \sum_{\begin{tiny}
\begin{array}{c}
r,q \in \ZH^3 \setminus \set{0} \\ 
r-q \neq 0
\end{array} \end{tiny}} \sin^2 \left( \frac{1}{2} q. \Th r \right) \nn \\ 
& \times \dfrac{  \e^{\l_2 \mu_5 \l_1} (-q_{\mu_5})}{8 \pi^2 k \mid q \mid^2} \dfrac{ \e^{\mu_2 \mu_4 \mu_2} (q_{\mu_4} - r_{\mu_4})}{8 \pi^2 k \mid q-r \mid^2 }\dfrac{\e^{\nu_2 \mu_3 \nu_1} r_{\mu_3} }{8 \pi^2 k \mid r \mid ^2}  \, ,
\end{align*}
The $2$-loop $L$ remains unchanged under the substitutions $r \ra -r $ and $q \ra -q$, and we obtain
\begin{align*}
L =& \dfrac{1}{2}  (2 \pi^2 i k)^2 \, N^3 \, \e^{\l_1 \mu_1 \nu_1} \e^{\l_2  \mu_2 \nu_2}  \sum_{\begin{tiny}
\begin{array}{c}
-r,-q \in \ZH^3 \setminus \set{0} \\ 
-r+q \neq 0
\end{array} \end{tiny}} \sin^2 \left( \frac{1}{2} (-q). \Th (-r) \right) \nn \\ 
& \times \dfrac{ \e^{\l_2 \mu_5 \l_1} q_{\mu_5}}{8 \pi^2 k \mid q \mid^2} \dfrac{ \e^{\mu_2 \mu_4 \mu_2} (r_{\mu_4} - q_{\mu_4})}{8 \pi^2 k \mid r-q \mid^2 }\dfrac{ \e^{\nu_2 \mu_3 \nu_1} (-r_{\mu_3}) }{8 \pi^2 k \mid r \mid ^2}      \nn \\
=& -L \, ,
\end{align*}
thus $L = 0$. This result holds true for every single $2$-loop which can be constructed from the other possible choices of vertices
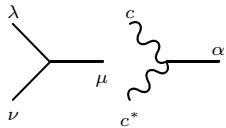
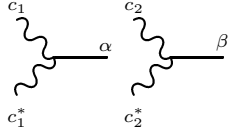
\begin{equation*}
\Diagram { \vertexlabel^{\l} \\ fd  \\
& f \vertexlabel_{\mu} \\ \vertexlabel_{\nu} fu  }  \quad
\Diagram { \vertexlabel^{c} \\ gd  \\ & f  \vertexlabel^{\a}\\ \vertexlabel_{c^{\ast}}  gu \\ }
\end{equation*}
or \\ 
\begin{equation*}
\Diagram { \vertexlabel^{c_1} \\ gd  \\
& f \vertexlabel^{\a}  \\
\vertexlabel_{c_1^{\ast}}  gu 
} \quad \Diagram { \vertexlabel^{c_2} \\ gd  \\
& f \vertexlabel^{\b} \\
\vertexlabel_{c_2^{\ast}}  gu 
}
\end{equation*} \\ \\
I. e. every single $2$-loop constructed from these vertices vanishes.

\begin{thm}  \label{second_expansion}
The $k^{-1}$-term in the loop expansion series of the partition function given in theorem \ref{partition} vanishes.
\end{thm}

This result corresponds to the classical, commutative one proved by E. Witten in \cite{witten}. Let $X$ be a two dimensional and smooth manifold without boundary and $M = X \times \TH$, where $\TH$ denotes the one dimensional circle. Witten proves in section 4.4 of \cite{witten} that the partition function of $M$ is independent of $k$. More precisely, he proves $Z(k) = \text{dim} \H_{X} $, where $\H_X$ is a Hilbert space which is related to the manifold $X$ by an Hamiltonian formalism which is outlined in \cite{witten}. Therefore, the coefficient of $k^{-1}$ in the Taylor expansion series of $Z(k)$ vanishes. In theorem \ref{second_expansion} we recovered this result for the special case that the manifold $X$ is a noncommutative 2-torus. \\
This can be motivated as follows. The deformation matrix $\Th$ is a skew-symmetric $3 \times 3$ matrix. Let $\Th$ be of the form
\begin{equation*}
\Th = \left( \begin{array}{ccc}
0 & \Th & 0 \\ 
-\Th & 0 & 0 \\ 
0 & 0 & 0
\end{array} \right),
\end{equation*} 
with $\Th \in \RH$. In this case $ \A_\Th \cong C^{\infty} \left( \TH^2_\Th \right) \ot C^{\infty}(\TH)$. Hence, we can think of the noncommutative 3-torus being a composition of a nonocmmutative 2-torus and $\TH$. Hence theorem \ref{second_expansion} is a noncommutative generalization of Witten's result for the case $X= \TH^2$.


\begin{thebibliography}{99}
\bibitem{jesper} A. Ambj\o rn, J. L. Petersen, \textit{Quantum Field Theory}, Lecture Notes, http://www.nbi.dk/~ambjorn/fieldthmain.ps
\bibitem{chern_simons} S.-S. Chern, J. Simons, \textit{Characteristic forms and geometric invariants}, Ann. of Math. (2) 99 (1974), 48-69
\bibitem{connes} A. Connes, \textit{Noncommutative Geometry}, Academic Press, San Diego, 1994 
\bibitem{connes_index} A. Connes, H. Moscovici, \textit{The Local Index Formula In Noncommutative Geometry}, GAFA, Vol. 5, No. 2 (1995), 174-243
\bibitem{connes_reconstruction} A. Connes,  \textit{On the Spectral Characterization of Manifolds}, arXiv:0810.2088v1, 2008
\bibitem{folland} G. B. Folland, \textit{Quantum Field Theory A Tourist Guide for Mathematicians}, AMS Mathematical Survey and Monographs, Vol. 149, Providence Rhode Island
\bibitem{iochum} D. Essouabri, B. Iochum, C. Levy, A. Sitarz, \textit{Spectral action on noncommutative torus}, J. Noncommut. Geom. 2 (2008), 53-123
\bibitem{pfante_chern_simons} O. Pfante,  \textit{Chern-Simons theory for noncommutative spaces}, in preparation 
\bibitem{pfante_sphere} O. Pfante,  \textit{Chern-Simons theory for $SU_q(2)$}, in preparation
\bibitem{voigt} J.-L. Loday, \textit{Cyclic Homology}, Springer, Berlin Heidelberg, 1992
\bibitem{witten} E. Witten, \textit{Quantum Field Theory and the Jones Polynomial}, Commun. Math. Phys. 121, 351-399 (1989)
\end{thebibliography}
\end{document}